\documentclass[11pt,leqno]{amsart}

\usepackage[all,cmtip]{xy}
\usepackage{a4,latexsym}
\usepackage[utf8]{inputenc}
\usepackage{amsfonts}
\usepackage[hidelinks]{hyperref}
\usepackage{amsxtra,amsmath,amsfonts,amscd, amssymb, mathrsfs, amsthm}
\usepackage{color, mathtools}
\usepackage{bbm}
\usepackage{enumitem}
\usepackage{xcolor}
\usepackage{graphicx}
\usepackage{tikz-cd}
\usepackage{braket}% \Set command
\usepackage[all]{xypic}
\usepackage[toc,page,title,titletoc,header]{appendix}
\usepackage[capitalize]{cleveref}

\numberwithin{paragraph}{section}

\setlist[enumerate]{label=\it{(\roman*)},
	ref=\it{(\roman*)}}

\newcommand{\C}{{\mathbb C}}
\newcommand{\R}{{\mathbb R}}

\newcommand{\Z}{{\mathbb Z}}

\newtheorem{theorem}{Theorem}[section]

\newtheorem{lemma}[theorem]{Lemma}
\newtheorem{lemma*}{Lemma}

\theoremstyle{definition}

\newtheorem{proposition&definition}[theorem]{Proposition\&Definition}
\newtheorem{lemma&definition}[theorem]{Lemma\&Definition}
\newtheorem{theorem&definition}[theorem]{Theorem\&Definition}

\newtheorem{example*}{Example}
\newtheorem{remark}{Remark}

\newtheorem{question*}{Question}

\newtheorem*{theorem2}{Theorem}

\def\p{{\partial}}
\def\S{\mathbf S}

\def\U{\mathcal U}
\def\W{\mathcal W}

\def\S{\mathbf S}

\def\ds{\displaystyle}

\numberwithin{equation}{section}

\begin{document}
	
	\title[ $L^p$-Boundedness of a Class of  Bi-Parameter  Pseudo-Differential Operators]{ $L^p$-Boundedness of a Class of  Bi-Parameter  Pseudo-Differential Operators }
	
	\author[J.~Cheng]{Jinhua Cheng}
	\address{J. Cheng,  School  of Mathematical Sciences, Zhejiang  University, 310058  Hangzhou, P. R. China }
	\email{chengjinhua@zju.edu.cn}
	
	%\author[W.~Gubler]{Walter Gubler}
	%\address{W. Gubler, Mathematik, Universit{\"a}t 
		%Regensburg, 93040 Regensburg, Germany}
	%\email{walter.gubler@mathematik.uni-regensburg.de}

	\begin{abstract}
In this paper,  we explore a specific class of bi-parameter pseudo-differential operators characterized by symbols $\sigma(x_1,x_2,\xi_1,\xi_2)$ falling within the product-type Hörmander {class}  
 $\S^m_{\rho, \delta}$. This classification imposes constraints on the behavior of partial derivatives of $\sigma$ with respect to both spatial and frequency variables. Specifically, we  demonstrate that for each multi-index $\alpha, \beta$, the inequality
 $| \partial_\xi^\alpha \partial_x^\beta \sigma(x_1,x_2,\xi_1,\xi_2)| \le C_{\alpha, \beta}(1+|\xi|)^m\prod_{i=1}^2  (1+|\xi_i|)^{-\rho|\alpha_i|+\delta|\beta_i|} $
is satisfied.  Our  investigation culminates in a rigorous analysis of the $L^p$-boundedness of such pseudo-differential operators, thereby extending the seminal findings of C. Fefferman from 1973 concerning pseudo-differential operators within the Hörmander class.
	\end{abstract}
	
	
	\maketitle
	
	\setcounter{tocdepth}{1}
	
	%\tableofcontents

%----------------------------------------------------------------------------------------
%	SECTION 1
%----------------------------------------------------------------------------------------	

\section{Introduction}
Consider a Schwartz function denoted by $f$. We define a pseudo-differential operator $T_{\sigma}$ as {follows:}  
\begin{equation}\label{pseudo operator}
T_{\sigma} f(x)=\int_{\mathbb{R}^n} e^{2\pi \mathrm{i} x\cdot\xi}\hat{f}(\xi) \sigma(x,\xi)d\xi,
\end{equation}
where $\hat{f}(\xi)$ represents the Fourier transform of $f$, and $\sigma(x,\xi)\in \mathcal{C}^{\infty}(\mathbb{R}^n\times \mathbb{R}^n)$ is referred to as the symbol. Of primary interest is the symbol class denoted $S^m_{\rho,\delta}$, commonly known as the Hörmander class. A symbol $\sigma(x,\xi)$ belongs to $S^m_{\rho, \delta}$ if it satisfies the following differential~inequalities:
\begin{equation}\label{Hormander class}
|\partial_\xi^\alpha \partial_x^\beta \sigma(x,\xi)|\le C_{\alpha,\beta} (1+|\xi|)^{m-\rho|\alpha|+\delta|\beta|},
\end{equation}
for all multi-indices $\alpha,~\beta$.

The $L^p$-boundedness of pseudo-differential operators,  defined  as in (\ref{pseudo operator}){--}(\ref{Hormander class}), has been a topic  of extensive investigation  in recent decades. Notably, the seminal works of Calderón and Vaillancourt \cite{CV71,CV72} established the $L^2$-boundedness of $T_\sigma$ for symbols $\sigma\in S^0_{0,0}$. Furthermore, Calderón and Vaillancourt showed that $T_{\sigma}$ remains bounded on $L^2$ when the symbol $\sigma$ belongs to $S^0_{\rho, \rho}$,  $0<\rho<1$, a class  known as the exotic symbol class. However, the boundedness results are not universal. For instance, consider the symbol  $\sigma(\xi)$ given by  the Fourier transform of  the Riemann singularity distribution $R(x) = e^{\frac{i}{|x|}} |x|^{-\frac{3}{2}}$, then $T_\sigma$
is not bounded on $L^p$ for $p \neq 2$. 
More recently, Wang \cite{W23} investigated a subclass of the exotic symbol class and demonstrated that pseudo-differential operators belonging to this subclass are bounded  on  $L^p$ for $0<\rho<1$.

The primary objective of this paper is to extend the following  theorem originally established by C. Fefferman in 1973 \cite{F73}.

\begin{theorem2}[Fefferman]
{Let} 
$\sigma(x, \xi) \in S^{m}_{\rho, \delta }$ with $0 \le \delta < \rho  <1$. Then
 \[
 \left\| T_\sigma f \right \|_{L^p(\R^n)} \le  C   \left\|  f \right \|_{L^p(\R^n)},   \quad 1<p< \infty, 
 \]
 when 
\[
   m \le -n(1-\rho) \left| \frac{1}{p}-\frac{1}{2} \right|.
\]
\end{theorem2}

To be more specific, we turn to the multi-parameter setting. Let $\sigma(x_1,x_2,\xi_1,\xi_2)\in \mathcal{C}^{\infty}(\mathbb{R}^{n_1}\times \mathbb{R}^{n_2} \times \mathbb{R}^{n_1}\times \mathbb{R}^{n_2})$, where $n=n_1+n_2$. We say $\sigma \in \S^m_{\rho,\delta}$ if it satisfies the  estimates

\begin{equation}\label{product Hormander class}
|\partial_\xi^\alpha \partial_x^\beta \sigma(x,\xi)|\le C_{\alpha,\beta} (1+|\xi|)^{m}\prod_{i=1}^2 \left( \frac{1}{1+|\xi_i|}\right)^{\rho|\alpha_i|-\delta|\beta_i|}
\end{equation}
for all multi-indices $\alpha,~\beta$.  Moreover, we define the bi-parameter H\"ormander class $BS^{m_1, m_2}_{\rho, \delta}$, we say 
$\sigma \in BS^{m_1, m_2}_{\rho,  \delta}$ if it satisfies the estimates:

\begin{equation}\label{product Hormander class}
|\partial_\xi^\alpha \partial_x^\beta \sigma(x,\xi)|\le  C_{\alpha,\beta} \prod_{i=1}^2 \left( \frac{1}{1+|\xi_i|}\right)^{\rho|\alpha_i|-\delta|\beta_i|- m_i }.
\end{equation}

Note that if $m=m_1+m_2$ and $m_1, m_2\le 0$, then we have

\[
 \S^m_{\rho,\delta} \subset  BS^{m_1, m_2}_{\rho,  \delta}.
\]

 {The}   classical theory of harmonic analysis may be described as around the Hardy--Littlewood maximal operator and its relationship with certain singular integral operators which commute with the classical one-parameter family dilations $\delta: x\rightarrow \delta x=(\delta x_1, \dots, \delta x_d), ~\delta >0$.  The multi-parameter theory, sometimes called product theory corresponds to a range of questions which are concerned with issues of harmonic analysis that are invariant with respect to a family of dilations $\delta: x\rightarrow \delta x=(\delta_1 x_1, \dots, \delta_d x_d), ~\delta_i >0, \linebreak i=1,\dots,d$. 
Such multi-parameter  symbol classes, associated with singular integral operators, pseudo-differential operators, and Fourier integral operators, have been the subject of extensive study by various authors. Notable contributions include works by Müller, Ricci, and Stein \cite{MRS95}, Yamazaki \cite{Y86}, Wang \cite{W22}, Chen, Ding, and Lu \cite{CDL20}, Huang and Chen \cite{HC21, XH21}, Hong, Zhang, and Lu \cite{HZ14, HL18, HLZ18, HZ17},
Muscalu,   Pipher,  Tao, and  Thiele \cite{MPTT04, MPTT06}, 
among others. The main result of this paper is the following

\begin{theorem}\label{multi-parameter pseduo}
 Suppose  $\sigma \in \S_{\rho, \delta}^{m}$,  $0\leq \delta<\rho<1$,  then 
 \[
 \left\| T_\sigma f \right \|_{L^p(\R^n)} \le  C   \left\|  f \right \|_{L^p(\R^n)},   \quad 1<p< \infty, 
 \]
 when 
\begin{equation}
   m \le -n(1-\rho) \left| \frac{1}{p}-\frac{1}{2} \right|.
\end{equation}
\end{theorem}

\begin{remark}
C. Fefferman originally proved the above theorem with symbols belonging to the classical Hörmander class $S^m_{\rho, \delta}$. Thus, the sharpness of the theorem follows from Fefferman's theorem, as $S^m_{\rho, \delta}\subset \S^m_{\rho,\delta}$.
\end{remark}
To prove Theorem \ref{multi-parameter pseduo}, it suffices to establish the following two estimates (\ref{key 1})-(\ref{key 2}), and then apply the complex interpolation theorem and consider the adjoint operator.
\begin{equation}\label{key 1}
\left \| T_\sigma f \right \|_{L^2(\R^n)} \le C \left \| f \right \|_{L^2(\R^n)},    \quad \sigma \in \S^0_{\rho, \delta}, \quad  0\le \delta <\rho <1,
\end{equation}
and 
\begin{equation}\label{key 2}
\quad \left \| T_\sigma f \right \|_{BMO(\R^n)} \le C \left \| f \right \|_{L^\infty(\R^n)}, \quad \sigma \in \S^{-n(1-\rho)/2}_{\rho, \delta}, \quad 0\le \delta < \rho <1.
\end{equation}
 
Here, $BMO(\mathbb{R}^n)$ denotes the class of functions of bounded mean oscillation defined by F. John and L. Nirenberg in \cite{JN61}. A locally integrable function $f$ on $\mathbb{R}^n$ belongs to $BMO$ if
\[
 ||f||_{BMO}=\sup_{Q}  \frac{1}{|Q|}\int_{Q}|f(x)-f_Q|dx <  \infty,
\]
where $Q$ is an arbitrary cube in $\mathbb{R}^n$ and $f_Q=\frac{1}{|Q|}\int_{Q}f(x)dx$.

{We} 
 will prove the $L^2$-boundedness of $T_\sigma$ with $\sigma$ of order $0$ in Section \ref{Section 2}, and $T_\sigma$ is bounded from $L^\infty$ to $BMO$ with $\sigma$ of order $-na/2$ in Section \ref{Section 3}. We primarily follow the proofs in \cite{S93} and \cite{F73} to establish estimates (\ref{key 1}) and (\ref{key 2}), respectively. However, a single Littlewood--Paley decomposition in the $\xi$-space is insufficient; we require a further cone decomposition to fully utilize the inequalities in (\ref{product Hormander class}).

\section{ {$L^2$}-Boundedness of $T_{\sigma}$ of Order 0 }\label{Section 2}
Since $ \S^0_{\rho,\delta} \subset  BS^{0, 0}_{\rho,  \delta}   \subset  BS^{0, 0}_{\delta,  \delta}$, it suffices to prove

\begin{lemma}\label{pseudo L^2 to L^2}
Suppose that $\sigma(x, \xi) \in BS^{0, 0}_{ \rho,  \rho}$, where  $0\le \rho<1$. Then the operator $T_\sigma$ defined in  (\ref{pseudo operator}) 
is bounded from $L^2(\R^n)$ to iteself.
\end{lemma}

\begin{proof}
 First we use the Cotlar--Stein lemma to show that the lemma is true in the case $\rho=0$.
 
By  Plancherel's theorem, we observe that it suffices to establish the $L^2$-boundedness of the operator $S$ defined by 
\[
Sf(x) = \int_{\R^n } e^{2\pi i x\cdot \xi } f(\xi) \sigma(x,\xi) d\xi. 
\]

{Notice} 
 that, in view of  the assumption of $\sigma$, the role of $x$ and $\xi$ in the above symbol class are perfectly symmetric. We choose a smooth non-negative function $\phi_i$ that is supported in the unit cube
\[
Q_1^i=\{x^i : ~|x_j^i| \le 1, ~~j=1,2,\dots, n_i \}, \qquad i=1,2
\]
and for which 
\[
\sum_{k^i \in \Z^{n_i}} \phi^i(x^i-k^i)=1.
\]

To construct such a $\phi^i$, simply fix any smooth, non-negative $\phi^i_0$ that equals $1$ on the cube $Q^{i}_{1/2} = 1/2 \cdot Q^i_1$ and is supported in $Q^i_1$. Noting that $\sum_{k^i \in. \Z^{ n_i} } \phi^i_0(x^i-k^i ) $ converges and is bounded away from $0$ for all $x\in \R^n$, we take
\[
\phi^i(x^i) =\phi^i_0(x^i) \left[ \sum_{l^i \in \Z^{n_i}} \phi_0^i(x^i-k^i) \right]^{-1}.
\] 

Next, let  $\vec k^i=(k^i, k'^i) \in \Z^{2n_i} =\Z^{n_i} \times \Z^{ n_i}$ denote an element of $\Z^{2n_i}$, and similarly write $\vec j^i=(j^i, j'^i ) $ for  another element of $\Z^{2 n_i}$. We set $\vec k=(\vec k^1,  \vec k^2 )$ and 
\[
\sigma_{\vec k} (x, \xi)=\left[ \prod_{i=1 }^2  \phi^i(x^i -k^i) \right] \sigma(x, \xi)  \left[ \prod_{i=1 }^2  \phi^i(\xi^i -k'^i) \right] = \phi(x-k)\sigma(x,\xi) \phi(\xi-k')
\] 
and 
\[
S_{\vec k}f(x) = \int_{\R^n } e^{2\pi  i x\cdot \xi } f(\xi)  \sigma_{\vec k} (x, \xi) d\xi. 
\]

Therefore, we have the decomposition
\[
S f(x) =\sum_{ \vec k\in \Z^{2n}} S_{\vec k} f(x).
\]

The main point is then to verify the almost-orthogonality estimates as follows:
\begin{equation}\label{almost orth 1}
||S^*_{\vec j} S_{\vec k}|| \le A \prod_{i=1}^2 (1+ |\vec j^i- \vec k^i |)^{ -2N_i}
\end{equation}
and
\begin{equation}\label{almost orth 2}
||S_{\vec k} S^*_{\vec j}|| \le A \prod_{i=1}^2 (1+ |\vec j^i- \vec k^i |)^{ -2N_i}
\end{equation}

Here, $||\cdot||$ denotes the $ L^2$ operator norm, $N_i$ is sufficiently large, and the bound $A$ is independent of $\vec k, \vec j$.   

Now, we can write 
\[
S^*_{\vec j } S_{\vec k} f(\xi) =\int_{\R^n} f(\eta) K_{\vec j, \vec k }(\xi, \eta) d \eta, 
\]
where
\[
K_{\vec j, \vec k }(\xi, \eta) = \int_{\R^n } \overline{\sigma}_{\vec j}(x,\xi ) \sigma_{\vec k}(x,\xi ) e^{2 \pi  i x\cdot (\eta- \xi) } dx.
\]

In the above integral, we integrate by parts, using the identities
\[
\prod_{i=1}^2  (I-\Delta_{x^i})^{N_i} e^{2\pi i x \cdot (\eta -\xi) } = \prod_{i=1}^2 (1+ 4 \pi ^2 |\eta^i- \xi^i|^2 )^{N_i} e^{2 \pi i x \cdot ( \eta- \xi) }.
\]

We also note that $\sigma_{\vec k}(x, \xi)$  and  $\sigma_{\vec j}(x, \eta)$  are given by
\[
\sigma_{\vec k} (x, \xi)=  \phi(x-k)\sigma(x,\xi) \phi(\xi-k'),\quad \sigma_{\vec j} (x, \xi)=  \phi(x-j)\sigma(x,\eta) \phi(\xi-j')
\] 
respectively, and so have disjoint $x$-support unless $\vec j^i-\vec k^i\in Q^i_1 $. These observations lead to the bounds
\[
\begin{array}{ lc} \displaystyle
|K_{\vec j, \vec k }(\xi, \eta)| \le  \prod_{ i=1} ^d \frac{ A_{N_i} \phi^i(\xi^i -\vec j'^i ) \phi^i(\eta^i -\vec k'^i )  }{(1+|\xi_i -\eta_i| )^{2N_i}}, ~~
\text{ if} ~~\vec j^i-\vec k^i\in Q^i_1, \quad i=1,2, 
\\\\ \displaystyle
|K_{\vec j, \vec k }(\xi, \eta)| =0, ~~~~~~~~~~~~~~~~~~~~~~~~~~~~~~~~~~~~~~~~~\text{otherwise}.
\end{array}
\]

Therefore, we have 
\[
\sup_{\xi} \int_{\R^n}|K_{\vec j, \vec k }(\xi, \eta)| d \eta <A \prod_{i=1}^2 (1+ |\vec j^i- \vec k^i |)^{ -2N_i}, 
\]
and
\[
\sup_{\eta} \int_{\R^n}|K_{\vec j, \vec k }(\xi, \eta)| d \xi < A \prod_{i=1}^2 (1+ |\vec j^i- \vec k^i |)^{ -2N_i},
\]
which implies  our desired estimate (\ref{almost orth 1}). Moreover, as we have noted, the situation is symmetric in $x$ and $\xi$, the same proof also shows the estimate (\ref{almost orth 2}). Now, it is only a matter of applying the Cotlar--Stein lemma; setting $N_i$ sufficiently large, we see 
\[
\sum_{\vec k \in \Z^n}  \prod_{i=1}^2 (1+ | \vec k^i |)^{ -2N_i} < \infty,  
\]
and as a result,  $S=\sum_{\vec k}S_{\vec k}$  is bounded from $ L^2(\R^n)$ to itself.

Now, we prove our Lemma \ref{pseudo L^2 to L^2}.

We start by defining a $C^\infty$ function $\varphi$ with compact support on $\mathbb{R}$, satisfying $\varphi(t)=1$ for $|t|\leq 1$ and $\varphi(t)=0$ for $|t|\geq 2$. For each $i=1,2$,  we set $\phi_0(\xi_i) =\varphi(|\xi_i|)$ and 
\[
\phi_{j_i}(\xi_i) = \varphi(2^{-j} |\xi_i| )-\varphi(2^{-j+1} |\xi_i| ), ~~ j_i\in \Z,~~j_i>0, ~~i=1,2,
\]
and 
\[
\phi_j(\xi) =\prod_{i=1}^2  \phi_{j_i} (\xi_i), \quad j\in \Z^2.
\]

Then, we define the partial operators 
\[
T_jf(x) =\int_{\R^n} e^{ 2\pi i x\cdot \xi } \hat{f}(\xi) \sigma_j(x,\xi) d\xi, \quad \sigma_j(x, \xi) =\sigma(x,\xi) \phi_j(\xi). 
\]

Let $\widehat{S_jf}(\xi)= \phi_j(\xi) \Hat{f}(\xi) $, and we have the decomposition of $T$
\begin{equation}\label{T decomposition}
 T=\sum_{j \ge 0} T_j =\sum_{j \ge 0} T S_j , \qquad \sum_{j \ge 0} = \prod_{i=1}^2 \sum_{j_i \ge 0}.
\end{equation}

It will be convenient to break the sum (\ref{T decomposition}) into two parts
\[
T= \sum_{ j~even}T_j + \sum_{j~odd }, \qquad ~\sum_{j~even} =\prod_{i =1}^2 \sum_{j_i \ge 0~ even },
\]
so that the summands in each parts have disjoint $\xi$-support; it suffices to prove the boundedness of each sum separately.

Let us consider the sum taken over the odd $j$. Note that 
\[
T_jT^*_k=TS_j(TS_k)^*= TS_jS^*_kT=0,\qquad j\neq k,
\]
because the supports of the multipliers corresponding to $S_j$ and $S_k$ are disjoint.
Next, we estimate $T^*_jT_k$, and we write
\[
T^*_jT_k f (x) = \int_{\R^n} K(x, y) f(y) dy,
\]
with 
\[
K(x,y) =\int_{\R^n\times\R^n\times \R^n}
\overline{\sigma}_k(z, \eta) \sigma_j(z, \xi) e^{2\pi i [\xi \cdot(z-y)-\eta \cdot(z-x) ] } dz d\eta d\xi.
\]

First, one carries integration by parts with respect to $z$-variable by writing
\[
\prod_{i=1}^2 \frac{(I- \Delta_{z_i} )^{N_i}}{(1+4 \pi^2|\xi_i -\eta_i|^2 )^{N_i}} e^{2\pi  (\xi-\eta )\cdot z }
=  e^{2\pi  i (\xi-\eta )\cdot z }.
\]

Next, one performs a similar process on the $\eta$-variable, beginning with 
\[
\prod_{i=1}^2 \frac{(I- \Delta_{\eta_i} )^{N_i}}{(1+4 \pi^2|x_i -z_i|^2 )^{N_i}} e^{2\pi i \eta\cdot (x-z )}
=  e^{2\pi  i \eta \cdot  (x-z) }.
\]

Finally, an analogous step is carried our for $\xi$-variable. If we take into account the differential inequalities for the symbols $\sigma_j$, and  the restrictions on their supports, we see that each order of differentiation in the $z_i$-variable gives us a factor of order
\[
(1+|\xi_i- \eta_i|)^{-1} \sim 2^{- \max\{k_i, j_i\}}
\]
for every factor of order 
\[
(1+|\xi_i|+ |\eta_i|)^{\rho} \sim 2^{\rho \max\{k-i, j_i \}}
\]
that may lose. As a result, the kernel $K$ is dominated by a constant multiple of 
\[
\prod_{ i=1}^2  2^{\max\{k_i, j_i \}(2\rho N_i-2N_i+2n_i )}  \int_{ \R^{n_i}} Q_i(x_i-z_i) Q_i(z_i -y_i) d z_i.
\]

Now, if we let $K_i(x_i,y_i) =\int_{ \R^{n_i}} Q_i(x_i-z_i) Q_i(z_i -y_i) d z_i  $, then
\[
\int_{\R^{n_i}} K_i(x_i, y_i) dy_i =\int_{\R^{n_i}} K_i(x_i, y_i) dx_i = \left(\int_{\R^{n_i} } (1+|z_i|)^{-2N_i}  \right)^2< \infty,
\]
if $2N_i >n_i$. Thus, we obtain 
\[
||T^*_jT_k ||\le A  \prod_{ i=1}^2 2^{\max\{k_i, j_i \}(2\rho N_i-2N_i+2n_i )}, \quad j\neq k,
\]
which implies that 
\[
||T_j^* T_k || \le \prod_{i=1}^2  \gamma_i(j_i) \gamma_i(k_i), \qquad j  \neq k,
\]
with $\gamma_i(j_i) =A\cdot 2^{ -\epsilon j_i}, \epsilon>0$, if we choose $N_i$ so large that $N_i>n_i(1-\rho)$.

In order to apply the Cotlar--Stein lemma, we need to show that the partial operators $T_j$  are uniformly bounded in the norm. To prove this, we set
\[
\Tilde{\sigma}_j(x, \xi)= \sigma_j(2^{-j\rho}x , 2^{j\rho} \xi),\quad 2^{-j \rho}x=(2^{j_1\rho }x_1, 2^{j_2\rho }x_2   ), ~~0\le \rho <1.
\]

Thus, $\Tilde{ \sigma}_j(x,\xi ) \in S^{0, 0}_{ \rho,  \rho}$ for $m_i=0, ~\rho_i=0$ for each $i=1,2,\dots,d$ uniformly in $j$ . Therefore,  the operator 
\[
\Tilde{T}_j f(x) =\int_{\R^n} e^{ 2\pi i x\cdot \xi } \hat{f}(\xi) \Tilde{\sigma}_j(x,\xi) d\xi
\]
is bounded on $L^2(\R^n)$. Next, define the scaling operators given by 
\[
\Lambda_j f( x) = f(2^{j\rho} x ) =f(2^{j_1\rho }x_1, 2^{j_2 \rho}x_2 ), 
\]
then, as is easily verified,
\[
T_j  =\Lambda_j \Tilde{T}_j \Lambda_j^{-1}.
\]

Now, $||\Lambda_j f||_{L^2} =\prod_{i=1}^2 2^{n_i j_i\rho/2} || f||_{L^2}$ and $||\Lambda^{-1}_j f||_{ L^2} =\prod_{i=1}^d 2^{-n_i j_i\rho/2} || f||_{L^2}$; so together with the $L^2$-boundedness of $\Tilde{T}_j$, we have 
\[
||T_j|| \le A, \quad \text{ uniformly in }~~j.
\]

We may therefore conclude that $\sum _{j ~odd }T_j$ is bounded from $L^2(\R^n)$ to itself, the sum $\sum_{j ~even}$ is treated similarly, and 
our Lemma \ref{pseudo L^2 to L^2} is {proved.}  

\end{proof}

\section{ {$L^p$}-Boundedness of {$T_{\sigma}$}  }\label{Section 3}
We make a further decomposition,  let $\xi=(\xi_1,\xi_2)\in \R^{n_1}\times\R^{n_2}$. 
Let 
 $\varphi$ be a smooth bump function on $\mathbb{R}$ such that
\begin{equation}\label{varphi}
\varphi(t)=1 \quad \textit{for} \quad |t|\le 1;  \quad \varphi(t)=0, \quad \textit{for}  \quad |t|\ge 2.
\end{equation}
Define 
\begin{equation}\label{phi_j}
\begin{array}{lc} \displaystyle
\phi_{j}(\xi) =\varphi(2^{-j}|\xi|) -\varphi(2^{-j+1}|\xi|) ,~~~j\in \mathbb{Z}, j>0;
 \quad 
\phi_{0} (\xi) = \varphi(|\xi|).
\end{array}
\end{equation}

 \begin{equation}
    \delta_\ell (\xi)= \varphi\left(2^{-\ell} \frac{|\xi_2|}{|\xi_1|}\right)-\varphi\left(2^{-\ell+1} \frac{|\xi_2|}{|\xi_1|}\right), \qquad \ell\in \Z.
\end{equation}
Note that  $\delta_\ell(\xi)$ has a support in the  cone  region,
\begin{equation}\label{cone}
    \Lambda_\ell=\{(\xi_1,\xi_2):2^{\ell-1}\le  \frac{|\xi_2|}{|\xi_1|}\le 2^{\ell+1}\}.
\end{equation}

By symmetry, we can always assume $\ell$ is a  non-negative integer.  Now for fixed $j$, we make a  cone decomposition in the frequency space, define partial operators
 \begin{equation}
  T_{\ell j} f(x)=  T_{\sigma_{\ell j}}f(x)=\int_{\R^n} e^{2\pi i x\cdot\xi} \Hat{f}(\xi) \sigma_{\ell j}(x,\xi)  d\xi, \quad 
     \sigma_{\ell j}(x,\xi)=\sigma(x,\xi) \phi_{j}(\xi) \delta_\ell(\xi).
 \end{equation}
Furthermore,  we define
 \begin{equation}
  T^\flat_{j} f(x)=  T_{\sigma^\flat_{j}}f(x)=\int_{\R^n} e^{2\pi i x\cdot\xi} \Hat{f}(\xi) \sigma^\flat_{ j}(x,\xi)  d\xi, \quad 
     \sigma^\flat_{ j}(x,\xi)=\sum_{\ell =j} ^\infty \sigma(x,\xi) \phi_{j}(\xi) \delta_\ell(\xi),
 \end{equation}
and 
 \begin{equation}
  T^\sharp_{j} f(x)=  T_{\sigma^\sharp_{j}}f(x)=\int_{\R^n} e^{2\pi i x\cdot\xi} \Hat{f}(\xi) \sigma^\sharp_{ j}(x,\xi)  d\xi, \quad 
     \sigma^\sharp_{j}(x,\xi)= \sum_{\ell =0}^j\sigma(x,\xi) \phi_{j}(\xi) \delta_\ell(\xi).
 \end{equation}

\subsection{A Key Lemma}
Let  a symbol $ \sigma(x,\xi)\in \S^m_{\rho, \delta}$, then we define its norm as  

\begin{equation}
    \left\|\sigma \right\|_\S=\sup_{|\alpha|\le k, |\beta|\le N} |\partial_\xi^\alpha \partial_x^\beta \sigma(x,\xi)|
    (1+|\xi|)^{-m}
    \prod_{i=1}^2
    \left( 1+|\xi_i|  \right)^{\rho|\alpha_i|-\delta |\beta_i|} ,~~k,N>n/2.
\end{equation}
 Let $r>0$ be a real number,  recall the definitions of $\phi_j(\xi)$ and $\delta_\ell(\xi)$,  define the partial~operators
 \begin{equation}
  T_{\ell j}^r  f(x)= T_{\sigma^r_{\ell j}}f(x)=\int_{\R^n} e^{2\pi i x\cdot\xi} \Hat{f}(\xi) \sigma^r_{\ell j}(x,\xi)  d\xi, \quad 
     \sigma^r_{\ell j}(x,\xi)=\sigma(x,\xi)\delta_\ell(\xi) \phi_{j}(r\xi),
 \end{equation}
 and
 \begin{equation}
T_j^rf(x) =    T_{\sigma^r_{ j}}f(x)=\int_{\R^n} e^{2\pi i x\cdot\xi} \Hat{f}(\xi) \sigma^r_{ j}(x,\xi)  d\xi, \quad 
     \sigma^r_{ j}(x,\xi)= \sum_{\ell\ge 0}\sigma^r_{\ell j}(x,\xi).
 \end{equation}

 \begin{lemma}\label{key lemma}
     Let the symbol $\sigma(x,\xi)$ be defined as (\ref{product Hormander class}),  and  $T_{\ell j}^r,~T_j^r$ defined as above,  then  we have
     \begin{equation}
         \begin{array}{cc}\ds
      \left\| T_j^r f \right\|_{L^\infty}\le  C   ||\sigma||_\S||f||_{L^\infty},
      \\\\ \ds
           \left\| T_{\ell j}^r f \right\|_{L^\infty}\le  C 2^{-n_1\ell/2}||\sigma||_\S||f||_{L^\infty}.
         \end{array}
     \end{equation}
     
Moreover, let $2^k\le r^{-1}<2^{k+1}$,  if $\sigma_0(x,\xi)=\sum_{j\le -k}\sigma^r_{j}(x,\xi)$, we have
\begin{equation}
            \left\| T_{\sigma_0} f \right\|_{L^\infty}\le  C ||\sigma||_\S||f||_{L^\infty}.
\end{equation}
 \end{lemma}
\begin{proof}
    
We denote $\Hat{\sigma}(x,y)=\int_{\R^n} e^{-2\pi  i y\cdot\xi} \sigma(x,\xi)d\xi $ throughout this paper. Now write 
\begin{equation}
    T_{j}^rf(x)=\int_{\R^n} f(y) \Hat{\sigma}_j^r(x,y-x)dy.
\end{equation}

We see that $|T_{j}^r f|\le ||\Hat{\sigma^r_j}(x,\cdot)||_{L^1}||f||_{L^\infty}$, where $||\Hat{\sigma_j^r}(x,\cdot)||_{L^1
}=\int_{\R^n}|\Hat{\sigma}^r_j(x,y)|dy$. Therefore, it suffices to show that $||\Hat{\sigma}^r_j(x,\cdot)||_{L^1}\le   C ||\sigma||_\S$,
~$||\Hat{\sigma}^r_{\ell j}(x,\cdot)||_{L^1}\le  C  2^{-n_1\ell/2} ||\sigma||_\S$ and 
$||\Hat{\sigma}_0(x,\cdot)||_{L^1}\le  C ||\sigma||_\S$ .
Let us consider $T_j^r$, let $b=(2^j r^{-1})^{a-1}$.  Applying the Cauchy--Schwartz inequality and Plancherel theorem we see
\begin{equation}
    \begin{array}{cc}\ds
\int_{|y|<b } |\Hat{\sigma^r_{ j}}(x,y)|dy\le  C  b^{n/2} \left( \int_{|y|<b } |\Hat{\sigma^r_{ j}}(x, y)|^2dy\right)^{\frac{1}{2}}
\le  C b^{n/2}\left( \int_{\R^n } |\sigma^r_{j}(x,\xi)|^2d\xi\right)^{\frac{1}{2}}
\\\\ \ds 
\le  C  ||\sigma||_\S \qquad (\textit{since $\sigma^r_{ j}$ lives in $|\xi|\sim 2^j r^{-1}$} ),
    \end{array}
\end{equation}
and 
\begin{equation}
    \begin{array}{cc}\ds
\int_{|y|\ge b } |\Hat{\sigma}^r_j(x,y)|dy\le C  b^{n/2-k} \left( \int_{|y|\ge b } |y|^{2k}|\Hat{\sigma}^r_j (x, y)|^2dy\right)^{\frac{1}{2}}
\\\\ \ds
\le C   b^{n/2-k} \left( \int_{\R^n } |\nabla_\xi^{k}  \sigma^r_{ j}(x,\xi)|^2d\xi\right)^{\frac{1}{2}}
\le C  ||\sigma||_\S \quad
\\\\ \ds(\textit{since $\sigma$ lives in $|\xi|\sim 2^j r^{-1}$}, k>n/2 ).
    \end{array}
\end{equation}

Thus, $||\Hat{\sigma^r_{j}}(x,\cdot)||_{L^1}\le C   ||\sigma||_\S$. 

Similarly, we can prove ~$||\Hat{\sigma}^r_{\ell j}(x,\cdot)||_{L^1}\le C2^{-n_1\ell/2} ||\sigma||_\S$ and 
~$||\Hat{\sigma}_0(x,\cdot)||_{L^1}\le C  ||\sigma||_\S$  once we   note that $\sigma_{\ell j}^r$  supported in the region $|\xi_1|\sim 2^{j-\ell},~|\xi_2|\sim 2^j$ and $\sigma_0$ supported in the region $|\xi|\le 1$. Therefore, we have proven  Lemma \ref{key lemma}.
\end{proof}

\subsection{Proof of the Main Theorem}\label{sec13}
Due to the complex interpolation theorem and the discussion of the adjoint operator, it suffices to prove the following Lemma \ref{BMO}.
\begin{lemma}\label{BMO}
Suppose  $\sigma(x, \xi) \in \S^{-n(1-\rho)/2}_{\rho, \delta}$,  $0 \le \delta< \rho<1$, 
then pseudo-differential operator $T_\sigma$  is a bounded operator from  $L^\infty(\R^n)$ to  $BMO(\R^n)$.
\end{lemma}

Consistent with the multi-parameter Fourier integral operators, and using the notation from the previous chapter, we define the partial operators
 \begin{equation}
  T_{j \ell } f(x)=  T_{\sigma_{j \ell }}f(x)=\int_{\R^n} e^{2\pi i x\cdot\xi} \widehat{f}(\xi) \sigma_{j \ell }(x,\xi)  d\xi, \quad 
     \sigma_{j \ell }(x,\xi)=\sigma(x,\xi) \phi_{j \ell }(\xi).
 \end{equation}
 \begin{equation}
  T_{ j} f(x)=  T_{\sigma_{j}}f(x)=\int_{\R^n} e^{2\pi i x\cdot\xi} \widehat{f}(\xi) \sigma_{j}(x,\xi)  d\xi, \quad 
     \sigma_{j}(x,\xi)=\sigma(x,\xi) \phi_{j}(\xi).
 \end{equation}

 \begin{lemma}\label{key lemma}
   Suppose  $\sigma(x,\xi) \in   \S^{-n(1-\rho)/2}_{\rho, \delta} $,      For the operator  $T_{j \ell}$ as defined above,  
there exists $\epsilon>0$  such that 
        \begin{equation}
      \left\| T_{j \ell } f \right\|_{L^\infty(\R^n) }\le  C \prod_{i=1}^{d-1} 2^{- \epsilon  \ell_i }  \left\| \sigma \right\|_\S \left\| f \right \|_{L^\infty(\R^n)},  
     \end{equation}
     \begin{equation}
      \left\| T_j f \right\|_{L^\infty(\R^n)}\le  C  \left\| \sigma \right\|_\S  \left\| f  \right \|_{L^\infty(\R^n)}, \quad j\ge 0.
     \end{equation}
 \end{lemma}
\begin{proof}
Let 
\[
\widehat{\sigma_{j \ell }}(x, y)=\int_{\R^n} e^{-2\pi  i y\cdot\xi} \sigma_{j \ell }(x,\xi)d\xi.
\]
Then we can express $T_{j \ell }$ as 
\begin{equation}
    T_{j\ell }f(x)=\int_{\R^n} f(y) \widehat{\sigma_{j \ell }}(x, y-x)dy.
\end{equation}
Then,
\[
\left \|T_{j \ell } f \right \|_{L^\infty(\R^n)} \le \left \|\widehat{\sigma_{j \ell }}(x,\cdot) \right \|_{L^1(\R^n)} \left\| f \right\|_{L^\infty(\R^n)}, 
\]
where 
\[
\left \| \widehat{\sigma_{j \ell}}(x,\cdot) \right \|_{L^1(\R^n)}=\int_{\R^n} \left |\widehat{\sigma_{j \ell}}(x, y) \right |dy.
\]
According to the definition, we have  $|\xi| \sim 2^{j}, |\xi_i |\sim 2^{j -\ell_i}, 1\le i \le d $.  
 
Let  $\U$ and  $\W$  be two disjoint subsets of  $ \{ 1,2, \dots, d\}$  such that $\U \cup \W = \{1, 2, \dots,d \}$. Define 
\[
Q= Q_{\U}= \{ y \in \R^n: |x_i| <  2^{- \rho j} \cdot 2^{\gamma  \ell_i }, i\in \U;  |x_i| \ge 2^{-\rho j}\cdot 2^{\gamma \ell_i }, i  \in \W \}.
\]
Thus, we can decompose $\R^n$  as $\R^n = \bigcup_{\U} Q_\U, $
where  $\bigcup_{\U}$ denotes the union over all subsets $U$ of  $\{1, 2, \dots, d\}$.
 Applying the Cauchy-Schwarz inequality and the Plancherel theorem, and assuming  $|\alpha_i| > n_i/2, i\in \W$, we have 
\begin{equation}
    \begin{array}{lc}\displaystyle
\int_{Q }  \left |\widehat{\sigma_{ j\ell  }}(x, y) \right |dy
=\int_{Q }  \prod_{i\in \W} |y_i|^{|\alpha_i|}  |y_i|^{ -|\alpha_i|}   \left  |\widehat{\sigma_{ j\ell  }}(x, y) \right |dy
\\\\ \displaystyle
\le  C  \left\{   \int_{Q}  \prod_{i\in \W} |y_i|^{ -2|\alpha_i|} dy\right \}^{\frac{1}{2}}
 \left \{  \int_{Q }  \prod_{i\in \W} |y_i|^{ 2|\alpha_i|}  \left  |\widehat{\sigma_{ j \ell}}(x, y) \right |^2dy\right \}^{\frac{1}{2}}
\\\\ \displaystyle 
\le C \prod_{i=1}^d  2^{-\rho j n_i/2} \cdot 2^{\gamma  \ell_i n_i /2}  \prod_{i \in \W} 2^{\rho j|\alpha_i |}  \cdot 2^{ -\gamma  \ell_i |\alpha_i|}  
 \left( \int_{\R^n } \left   | \prod_{i\in \W} \p_{\xi_i}^{\alpha_i}\sigma_{j \ell }(x,\xi) \right |^2d\xi\right)^{\frac{1}{2}}
\\\\ \displaystyle 
\le C \prod_{i=1}^d  2^{-\rho j n_i/2} \cdot 2^{\gamma \ell_i n_i /2}  \prod_{i \in \W} 2^{\rho j|\alpha_i |}  \cdot 2^{ -\gamma \ell_i |\alpha_i|}   \cdot 
\prod_{i =1}^d 2^{j_in_i /2 }  \cdot 2^{-j n(1-\rho)/2 } 
\prod_{i \in \W}  2^{-j_i \rho|\alpha_i|}\left \| \sigma \right\|_\S
\\\\ \displaystyle
\le  C \prod_{i \in \U}2^{-\ell_i (1-\gamma)n_i /2}  \cdot   \prod_{i \in \W}2^{-\ell_i \left [ (\rho-\gamma)(n_i /2 -|\alpha_i|)   \right] }  
 \left \|\sigma \right\|_\S.
    \end{array}
\end{equation}

Therefore, if we let 
\[
 0 < \rho < \gamma<1 , \quad |\alpha_i|> n_i/2, \quad i \in \W,
\]
then there exists	$\epsilon>0$ such that 
\[
\int_{\R^n }  \left |\widehat{\sigma_{ j\ell  }}(x, y) \right |dy \le  \prod_{i=1}^{d-1} 2^{- \epsilon  \ell_i }  \left\| \sigma \right\|_\S.
\]
Thus, Lemma \ref{key lemma} holds.

\end{proof}
With Lemma \ref{key lemma} in hand, we now proceed to prove Lemma \ref{BMO}.
\begin{proof}
Fix a function $f\in L^\infty(\R^n)$,  and let   $Q\subset \R^n$ be a cube with side length $r$,  and center  $x_0$.  We will show that 
\begin{equation}\label{aim}
    \frac{1}{|Q|}\int_{\R^n} |T_\sigma f(x) -(T_\sigma f)_Q|dx\le  C \left\| f \right\|_{L^\infty (\R^n)}.
\end{equation}
Decompose $\sigma$ into two parts $\sigma = \sigma^0 +\sigma^1$ such that $\sigma^0$ is supported on  $|\xi| \le 2 r^{-1}$,  $\sigma^1$ is supported on   $|\xi| \ge  r^{-1}$, 
\[
\sigma^0(x, \xi) =\sigma(x, \xi)\phi_0( r \xi). 
\]
Direct calculation shows that
\[
\partial_{x_i}T_{\sigma^0}f(x)=T_{\sigma'}f(x)
\]
where
 \[
\sigma'(x,\xi)=\partial_{x_i}\sigma^0(x,\xi)+2\pi i \xi_i\sigma^0(x,\xi).
\]
Since    $|\xi| \le 2 r^{-1} $,    $\left\| \sigma' \right \|_\S\le  C r^{-1} \left\| \sigma  \right \|_\S$.  by Lemma \ref{key lemma},  
\begin{equation}
\begin{array}{lc}\displaystyle 
    \left\|  \partial_{x_i} T_{\sigma^0} f \right\|_{L^\infty(\R^n)}
    \le C r^{-1} \left\| \sigma \right \|_\S \left\| f \right\|_{L^\infty(\R^n )}.
\end{array}
\end{equation}
Thus, there exists a constant  $a_Q$  such that $| T_{\sigma^0} f(x)-a_Q|$ is bounded on  $Q$, Therefore,  
\begin{equation}\label{estimate j<0}
    \frac{1}{|Q|}\int_Q\left| T_{\sigma^0}f(x)-a_Q\right|dx\le   \left\| \sigma \right\|_\S  \left\| f \right\|_{L^\infty(\R^n)}.
\end{equation}

Now, consider the case of $T_{\sigma^1}$,  Fix a smooth function on $\R^n$, 
 $\lambda$,  such that  $0\le \lambda(x)\le 10$ and  $\lambda(x)\ge1$ on $Q$,  
  $\widehat{\lambda}(\xi)$  is supported on   $|\xi|\le r^{- \rho } $.  Then
\begin{equation}\label{I_1+I_2}
    \lambda(x)T_{\sigma^1}f(x)=T_{\sigma^1}(\lambda f)(x)+[\lambda,T_{\sigma^1}]f(x)=I_1+I_2.
\end{equation}
First we consider $I_1$, 
\[
T_{\sigma^1}(\lambda f)=(T_{\sigma^1} \cdot G_{-n(1-\rho)/4})\cdot (G_{n(1-\rho)/4}(\lambda f)),\quad \widehat{G_\alpha f}(\xi)=(1+|\xi|^2)^{-\alpha} \widehat{f}(\xi).
\]
Note that  $T_{\sigma^1}\cdot G_{-n(1-\rho)/4}$  is a pseudo-differential operator, and the symbol  $\sigma^1(x,\xi)(1+|\xi|^2)^{n(1-\rho)/4}\in \S^0_{\rho,\delta}$, Therefore, by the results from the previous section,
   $T_{\sigma^1}\cdot G_{-n(1-\rho)/4}$  is bounded on 	$L^2(\R^n)$.
\begin{equation}
\begin{array}{lc}\displaystyle
    \left \|T_{\sigma^1}(\lambda f) \right\|_{L^2(\R^n)}^2\le C  \left\| \sigma \right\|_\S^2 \left\| G_{n(1-\rho)/4}(\lambda f)  \right \|_{L^2(\R^n)}^2 
    \\\\ \displaystyle 
    \le C  \left\| \sigma \right\|^2_\S  \left \| f  \right\|_{L^\infty(\R^n) }^2 \left \| G_{n(1-\rho)/4}\lambda \right\|^2_{L^2(\R^n)},
    \end{array}
\end{equation}
 Since 
\[
\begin{array}{lc}\displaystyle
 \left \| G_{n(1-\rho)/4}\lambda \right\|^2_{L^2(\R^n)} = \int_{\R^n} \left  | (1+|\xi|^2)^{-n(1-\rho)/4} \cdot \widehat{ \lambda}(\xi )\right |^2  d\xi 
 \\\\ \displaystyle
 \le C  \int_{|\xi| \le r^{-\rho}} (1+|\xi|^2)^{-n(1-\rho)/2}  d\xi \le \C |Q|.
 \end{array}
 \]
 Therefore, we have 
 \[
   \left \|T_{\sigma^1}(\lambda f) \right\|_{L^2(\R^n)}^2\le C 
   \left\| \sigma \right\|^2_\S  \left \| f  \right\|_{L^\infty(\R^n) }^2 |Q|,
 \]
 and 
 \begin{equation}\label{lambda f}
 \frac{1}{|Q|} \int_{ Q } \left| T_{\sigma^1} (\lambda f)(x)  \right|dx \le \left(   \frac{1}{|Q|} \int_{ Q } \left| T_{\sigma^1} (\lambda f)(x)  \right|^2 dx \right)^{\frac{1}{2}}
 \le C  \left\| \sigma \right\|_\S  \left \| f  \right\|_{L^\infty(\R^n) }.
 \end{equation}
Simple calculations yield that
 \[
 \left[\lambda, T_{\sigma^1} \right] f  = T_{\theta} f,   \quad \theta(x, \xi ) = \int_{\R^n} e^{2 \pi \i x \cdot  \eta } \widehat{\lambda}(\eta) \left[ \sigma^1(x, \xi) - \sigma^1(x, \xi-\eta) \right] d \eta.
 \]
Now,  decompose $\theta$, 
 \[
 \theta_{j \ell} (x, \xi) = \theta(x, \xi) \phi_{j \ell}( r \xi), \quad  \phi_{j \ell}( r \xi)= \prod_{i=1}^d \phi_{j_i}(r|\xi_i| )
 \]
Applying the differential mean value theorem and noting that $\widehat{\lambda}(\eta)$ is supported on  $|\eta| \le r^{-\rho} $, 
it can be seen that
 \[
 \left \|  \theta_{j \ell}  \right \|_\S \le  C  2^{- (j - \ell_M)\rho}   \left \|  \sigma \right \|_\S \le  C 2^{- \epsilon' (j - \ell_M)}   \left \|  \sigma \right \|_\S, \quad \epsilon'>0.
 \]
According to Lemma \ref{key lemma}, we have
 \[
  \left\|  T_{\theta_{j \ell}} f   \right\|_{L^\infty(\R^n)} \le C  2^{- \epsilon' (j - \ell_M)} \cdot  \prod_{i=1}^{d-1} 2^{- \epsilon \ell_i}  \left\| \sigma \right\|_\S
  \left\| f \right \|_{L^\infty(\R^n)},  
 \]
 take $\epsilon'  < \epsilon$, then 
 \begin{equation}\label{T theta}
  \left\|    \left[\lambda, T_{\sigma^1} \right] f      \right\|_{L^\infty(\R^n)} 
=    \left\|  T_{\theta} f   \right\|_{L^\infty(\R^n)} 
  \le C   \left\| \sigma \right\|_\S \left \| f \right \|_{L^\infty (\R^n)},  
 \end{equation}
Combining the estimates (\ref{lambda f}) and (\ref{T theta}), we obtain
 \[
 \frac{1}{|Q|} \int_{\R^n} \left\|   \lambda \cdot T_{\sigma^1}f(x)  \right\| dx \le  C   \left\| \sigma \right\|_\S  \left \| f \right \|_{L^\infty(\R^n)}.
 \]
Moreover, since  $|\lambda(x)|\ge 1$ when  $x \in Q$,   it follows that 
\[
\frac{1}{|Q|}\int_{Q} | T_{\sigma^1} f(x)|dx\le C  \left\| \sigma  \right \|_\S  \left \| f \right \|_{L^\infty(\R^n)}.
\]
By combining the above estimate with (\ref{estimate j<0}), we obtain (\ref{aim}), thus proving Lemma \ref{BMO} and Theorem \ref{multi-parameter pseduo}.
\end{proof}

\addcontentsline{toc}{section}{Acknowledgements}		
%\section*{Acknowledgements}	
	
	%\bibliography{references}
	%\bibliographystyle{alpha.bst}
\end{document}